\documentclass[final]{siamltex}
\usepackage{amsfonts}
\usepackage{graphicx}

\newcommand{\ba}{\begin{array}}
\newcommand{\ea}{\end{array}}
\newcommand{\bdm}{\begin{displaymath}}
\newcommand{\edm}{\end{displaymath}}
\newcommand{\be}{\begin{equation}}
\newcommand{\ee}{\end{equation}}

 \newcommand{\bea}{\begin{eqnarray}}
   \newcommand{\eea}{\end{eqnarray}}
 \newcommand{\beax}{\begin{eqnarray*}}
 \newcommand{\eeax}{\end{eqnarray*}}

\title{Essential spectral equivalence via multiple step preconditioning and applications to  ill conditioned Toeplitz matrices }

\author{D. Noutsos\footnotemark[1] \and  S. Serra-Capizzano\footnotemark[2]\and P. Vassalos\footnotemark[3] \footnotemark[4]}

\begin{document}
\pagenumbering{arabic} \setcounter{page}{1}

\maketitle
\begin{abstract}
In this note, we study   the fast solution of Toeplitz linear
systems with coefficient matrix  $T_n(f)$, where the generating
function $f$ is nonnegative and has a unique zero at zero of any
real positive order $\theta$. As preconditioner we choose a matrix
${\tau}_n(f)$ belonging to the so-called $\tau$ algebra, which is
diagonalized by the sine transform associated to the discrete
Laplacian. In previous works, the spectral equivalence of the
matrix sequences $\{{\tau}_n(f)\}_n $ and $\{T_n(f) \}_n$ was
proven under the assumption that the order of the zero is equal to
$2$: in other words the preconditioned matrix sequence
$\{{\tau}^{-1}_n(f)T_n(f) \}_n $ has eigenvalues, which are
uniformly away from zero and from infinity. Here we prove a
generalization of the above result when $\theta<2$. Furthermore,
by making use of multiple step preconditioning, we show that the
matrix sequences $\{{\tau}_n(f)\}_n $ and $\{T_n(f) \}_n$   are
essentially spectrally equivalent for every $\theta>2$, i.e., for
every $\theta>2$, there exist $m_\theta$ and a positive interval
$[\alpha_\theta,\beta_\theta]$ such that all the eigenvalues of
$\{{\tau}^{-1}_n(f)T_n(f) \}_n $ belong to this interval, except
at most $m_\theta$ outliers larger than $\beta_\theta$. Such a
nice property, already known only  when  $\theta$ is an even
positive integer greater than 2, is coupled with the fact that the
preconditioned sequence has an eigenvalue cluster at one, so that
the convergence rate of the associated preconditioned conjugate
gradient method is optimal. As a conclusion we discuss possible
generalizations and we present selected numerical experiments.

\end{abstract}



\maketitle

\renewcommand{\thefootnote}{\fnsymbol{footnote}}
\footnotetext[1]{Department of Mathematics, University of
Ioannina, GR45110 Greece. ({\tt dnoutsos@uoi.gr})}
\footnotetext[2]{Department of Science and high Technology,
University of Iunsubria, Como, Italy. ({\tt
stefano.serrac@uninsubria.it})} \footnotetext[3]{Department of
Informatics, Athens University of Economics and Business, GR10434
Greece. ({\tt pvassal@aueb.gr})} \footnotetext[4]{This research
has been co-financed by the European Union  (European Social Fund
- ESF) and Greek national funds through the Operational Program
"Education and Lifelong Learning" of the National Strategic
Reference Framework (NSRF) Research Funding Program THALES:
Investing in knowledge society through the European Social Fund.}

\section{Introduction}
\newcommand{\ui}{\mathbf{i}}

Our goal is to design and  analyze a preconditioning technique
for the fast solution of a Toeplitz system with $n\times n$
coefficient matrix $T_n(f)$, where $f$ is a given function having
a unique zero at zero of positive order $\theta$: the entry
$(j,k)$, $1\le j,k\le n$, of the matrix $T_n(f)$ is the $l$-th
Fourier coefficient of $f$ with $l=j-k$ and
\[
a_l = \frac1{2\pi}\int_0^{2\pi}
f(t)e^{-{\ui} l t}\ dt.
\]
The preconditioner is chosen in the so-called $\tau$ algebra which
is the set of all real symmetric matrices diagonalized by the sine
transform associated to the discrete Laplacian (see
(\ref{sine_matrix})): the preconditioner is chosen to have as
eigenvalues a uniform sampling of the symbol $f$ and is denoted by
${\tau}_n(f)$.

We study  the spectrum of the matrix sequences $\{{\cal A}_n\}_n$
with ${\cal A}_n={\tau}^{-1}_n(f)T_n(f)$  with the goal of
localizing the eigenvalues and understanding the asymptotic
behavior. We recall that the study of such a matrix sequence gives
precise information on the convergence speed of the related
preconditioned Conjugate Gradient (PCG) method and the associated
preconditioning strategy can be used in connection with multigrid
schemes: see \cite{cmame1} for the use of fast Toeplitz
preconditioning in the context of a multigrid method for a
Galerkin isogeometric analysis approximation to the solution of
elliptic  partial differential equations. Furthermore, the
analysis of the sequence $\{{\cal A}_n\}_n$ can be helpful in the
development of new approaches as the Jacobi-Davidson method in the
context of eigenvalue problems.

 The problem of understanding the spectrum of $\{{\cal A}_n\}_n$ has been
 extensively studied in the literature, see for example  \cite{M.Ng:book}, \cite{R.Cha/X.Jin:book}
 and references therein,  when the generating function  has  zeros of  even multiplicity.
 Here, to the best of our knowledge,  it is the first time that the general case is considered.
 For the sake of simplicity, we restrict our attention to the case where $f$ has a unique zero at
 zero with positive order $\theta$: it is worthwhile observing that in such a context the band Toeplitz
 preconditioning cannot lead to spectrally equivalent or essentially spectrally equivalent sequences, just because a
 nonnegative trigonometric polynomial cannot have zeros of non even order (see \cite{S.Ser:95} for a discussion on the subject).

The spectral relations  between  Toeplitz and  $\tau$ matrices
have  been analyzed by many researchers. Specifically, the
spectral properties of this algebra are  investigated in
\cite{D.Bin/M.Cap:83} and its approximation features, in
connection with Toeplitz structures, are treated in
\cite{D.Bin/F.DiB:93}. In  \cite{F.Ben/G.Fio/S.Ser:93},
\cite{F.Ben:95}, \cite{T.Huc:96} several $\tau$ preconditioning
techniques are studied, while in \cite{F.Ben/S.Ser:99} the
spectral properties of $\tau$ preconditioned matrices are
considered in detail.

In the quoted literature, in order to perform a theoretical
analysis,  the authors assumed that the generating function has
zeros of even orders.    The novel contribution of this work
relies on the  relaxation of this assumption. Precisely,  we study
the spectral properties of  the matrix sequence $\{{\cal
A}_n\}_n$, by dividing the analysis into two steps: first we
consider the case where the order of the zero is $\theta\in (0,2]$
and then, by using a multiple step preconditioning, we consider
the case  $\theta>2$, which is somehow reduced to the first case.

The paper is organized as follows. \S\ref{sec:prel} contains the
necessary preliminary definitions: in particular we define the
$\tau$ algebra, the preconditioner, and the notion of (essential)
spectral equivalence. In \S\ref{sec:tools} we briefly describe the
tools we use i.e., a special block Toeplitz operator and the
multiple step preconditioning.   The main theoretical statements
of this  paper are  presented and proved  in \S\ref{sec:main}
and concern the assumptions which leads to the (essential)
spectral  equivalence between ill conditioned Toeplitz sequences
and  the associated $\tau$ preconditioners. In \S\ref{sec:numexp}
we report and critically discuss  various numerical experiments,
while  Section \ref{sec:final}  is devoted to concluding remarks
and to potential future extensions.

\section{Preliminaries}\label{sec:prel}

Let $f $ be a nonnegative even function having, for  simplicity, a
single zero at the point $x_0=0$ of order $\theta\in
\mathbb{R}^+$, where $\mathbb{R}^+$ is the set of positive real
numbers,    and let $\{T_n(f)\}_n$ be the related Toeplitz matrix
sequence.  Then, the $\tau$ sequence $\{P_n\}_n$ constructed  as
\be\label{prec_def} P_n=\tau_n(f)=S_n \textup{diag}(f(w^{[n]}))S_n
\ee is considered as a preconditioning sequence for
$\{T_n(f)\}_n$:  here $w^{[n]}$ is the  $n$ dimensional vector
with entries $w_i^{[n]}=\frac{\pi i}{n+1}$, $i=1,\ldots,n$, $S_n$
is the sine-transform matrix defined as \be\label{sine_matrix}
({S_n})_{ij}=\sqrt{\frac{2}{n+1}}\left(\sin(jw^{[n]}_i)\right)_{i,j=1}^{n},
 \ee
and $\textup{diag}(f(w^{[n]}))$ is the diagonal matrix  having as
diagonal entries, the sampling of the values of $f$ on the
specific discretization $w^{[n]}$. Obviously,  $P_n$ is always
positive definite and the same holds true whenever the zero (or
zeros) of the generating function does not coincide with the
discretization points $w^{[n]}$. Under this assumption, the
developed theory of the next section holds unaltered.   We mention
that the $\tau$ matrices constructed in this way  are not the
``Frobenius optimal'' $\tau$ preconditioners
\cite{D.Bin/F.DiB:93}:  they coincide with  the ``natural'' $\tau$
preconditioner only if $f$ is a trigonometric polynomial  (see
e.g. \cite{S.Ser:99}).

The main goal  of this work is to show that the sequences  of
matrices $\{P_n=\tau_n(f)\}_n$ and $\{T_n(f)\}_n$ are spectrally
equivalent  whenever the symbol $f$ has a unique zero at zero of
order $\theta\le 2$. Moreover, when $\theta>2$, the  essential
spectral equivalence between these two sequences of matrices can
be proven. The notions of  spectral and essential spectral
equivalence are reported below.

\begin{definition}\label{def:ess-spectral-equiv}
Given two sequences of positive
definite matrices, $\{A_n\}_n$ and $\{P_n\}_n$  we say that they are spectrally equivalent iff the spectrum
$\{\sigma(P_n^{-1}A_n)\}_n$ of $\{P_n^{-1}A_n\}_n$ belongs to a positive
interval $[\alpha,\beta]$, where $\alpha,  \beta$ are constants independent of $n$ with $0<\alpha\le \beta<\infty$.
We say that the sequences $\{A_n\}_n$ and $\{P_n\}_n$ are essentially
spectrally equivalent iff $\{\sigma(P_n^{-1}A_n)\}_n$ is contained in $[\alpha,\beta]$, with at most a constant number of outliers greater than $\beta$.
\end{definition}

\section{Tools}\label{sec:tools}

As we have mentioned in the introduction, the theoretical tools
that are  used in the literature to prove the (essential) spectral
equivalence between ill conditioned Toeplitz sequences generated
by a symbol having a zero of even order at zero,  and proper
matrix algebra sequences, cannot be applied in our case. Thus, the
main tools for proving  our arguments will be  results coming from
block Toeplitz matrices, properties on Schur complements, the
flexibility of the Rayleigh quotient in the min-max, max-min
characterizations of eigenvalues of Hermitian matrices, and a
general theorem on the multiple step preconditioning. A brief
overview  of  them is presented   in the next subsections.


\subsection{A special block Toeplitz operator}\label{sec:tool1}

Regarding block Toeplitz matrices, we remind that if  $F(t)$ is a
$2\times 2$ matrix-valued function of the form \bdm
F(t)=\left(\begin{array}{cc} f_1(t)&f_2(t)\\
f_3(t)&f_4(t)\end{array}\right),\edm then, the matrix \bdm
B_{2n}(F)=\left(\begin{array}{cc} T_n(f_1)&T_n(f_2)\\
T_n(f_3)&T_n(f_4)\end{array}\right) \edm is a block Toeplitz
matrix. Note that the resulting  structure, and consequently   its
spectral properties, are quite different from the ones of the
scalar and multi-level Toeplitz forms, but there is a strong link
with the one-level Toeplitz matrices generated by a matrix-valued
function. In fact, there exists a simple permutation $\Pi$ such
that
\[
T_n(F)=\Pi B_{2n}(F) \Pi^T
\]
and hence the spectrum of $B_{2n}(F)$ coincides with that of
$T_n(F)$.  Furthermore, from the analysis in
\cite{M.Mir/P.Til:00}, it is known that $T_n(F)$ (and so
$B_{2n}(F)$) is positive semidefinite, whenever the generating
function $F$ is positive semidefinite and, in addition, $T_n(F)$
is positive definite if the minimal eigenvalue of $F$ is not
identically zero; see \cite{S.Ser:99:BIT}. We will use these
properties later on in our main derivations in Theorem
\ref{th:main}.

\subsection{The multiple step preconditioning}\label{sec:tool2}

Consider a linear system with a positive definite coefficient
matrix $A_n$  and suppose we have a chain of positive definite
preconditioners  $P_n^{(0)},\ldots,P_n^{(l)}$ such that
$P_n^{(j+1)}$ is an optimal preconditioner for $P_n^{(j)}$ (i.e.
we have essential spectral equivalence between the two sequences),
$j=0,\ldots,l-1$, $P_n^{(0)}=A_n$.

Then, a natural approach is to use a PCG at the external level
with  coefficient matrix $A_n$ and preconditioner $P_n^{(1)}$.
Furthermore, for all the auxiliary linear systems involving
$P_n^{(1)}$, we use again a PCG method with $P_n^{(2)}$ as
preconditioner and so on. Given the optimal convergence rate of
all the considered PCG methods, it is easy to see that the global
procedure is optimal, but the scheme could lose efficiency already
for moderate values of $l$. Therefore we would like to use the
final preconditioner $P_n=P_n^{(l)}$ directly on the original
system, with coefficient matrix $A_n$. The following theorem gives
a theoretical ground for this choice, showing that $P_n$ is an
optimal preconditioner of $A_n$ if, for every $j=0,\ldots,l-1$,
the matrix  $P_n^{(j+1)}$ is an optimal preconditioner for
$P_n^{(j)}$: this result will be used later on in Theorem
\ref{th:main}.

\begin{theorem}\label{th:gen}
Let $A_n,P_n$ two positive definite matrices of size $n$. Assume there exist positive definite matrices $P_n^{(0)},\ldots,P_n^{(l)}$,
positive numbers $\alpha_0,\ldots,\alpha_{l-1}$, $\beta_0,\ldots,\beta_{l-1}$, integer numbers
$r_0^{-},\ldots,r_{l-1}^{-}$, $r_0^{+},\ldots,r_{l-1}^{+}$, $l\ge 1$, such that
\begin{itemize}
\item $P_n^{(0)}=A_n$, $P_n^{(l)}=P_n$, $\alpha_j\le \beta_j$, $j=1,\ldots,l-1$,
\item the eigenvalues $\left(P_n^{(j+1)}\right)^{-1} P_n^{(j)}$ belong to the interval $[\alpha_j,\beta_j]$ with the exception
of  $r_j^{-}$ outliers less than $\alpha_j$ and of $r_j^{+}$ outliers larger than $\beta_j$, $j=0,\ldots,l-1$.
\end{itemize}

Then, all the eigenvalues of $P_n^{-1} A_n$ belong to the interval $[\alpha,\beta]$, $\alpha=\prod_{j=0}^{l-1} \alpha_j$,
$\beta=\prod_{j=0}^{l-1} \beta_j$, with the exception of $r^-$ outliers less than $\alpha$ and $r^+$ outliers larger that $\beta$,
$r^-=\sum_{j=0}^{l-1} r_j^{-}$, $r^+=\sum_{j=0}^{l-1} r_j^{+}$.
\end{theorem}
\proof
Let
\[
\lambda_1\ge \lambda_2 \ge \cdots \ge \lambda_n>0
\]
be the eigenvalues of  $P_n^{-1} A_n$  and let $k\in
\{r^++1,\ldots,n-r^-\}$.  Then it suffices  to prove that
$\lambda_k \in  [\alpha,\beta]$. To this end, we make use of
min-max and max-min characterization of the eigenvalues of
Hermitian matrices and we can use this argument since $P_n^{-1}
A_n$ is similar to the Hermitian (indeed positive definite) matrix
$P_n^{-1/2} A_nP_n^{-1/2}$. Hence
\begin{equation}\label{max-min}
\lambda_k= \max_{\textup{dim}(V)=k} \min_{v\in V, v\neq 0} \frac{v^* A_n v}{v^* P_n v},
\end{equation}
\begin{equation}\label{min-max}
\lambda_k= \min_{\textup{dim}(V)=n+1-k} \max_{v\in V, v\neq 0} \frac{v^* A_n v}{v^* P_n v}.
\end{equation}
Now, for every $j=0,\ldots,l-1$, set
\[
Q_j=\left(P_n^{(j+1)}\right)^{-1/2} P_n^{(j)} \left(P_n^{(j+1)}\right)^{-1/2},
\]
consider the subspaces $F_j(-)$ spanned by the $r_j^-$
eigenvectors of $Q_j$  related to the eigenvalues which are  less
than $\alpha_j$,and  $F_j(+)$ spanned by the $r_j^+$ eigenvectors
of $Q_j$ for which the correspondent eigenvalues are greater than
$\beta_j$. Then, we define:
\begin{equation}\label{important subspace -}
L(-)=\bigcap_{j=0}^{l-1} \left(P_n^{(j+1)}\right)^{-1/2} [F_j(-)]^\perp,
\end{equation}
\begin{equation}\label{important subspace +}
L(+)=\bigcap_{j=0}^{l-1} \left(P_n^{(j+1)}\right)^{-1/2} [F_j(+)]^\perp.
\end{equation}
By the assumptions, the subspaces $F_j(-)$ and  $F_j(+)$   have
dimension $r_j^-$ and $r_j^+$, respectively. Thus,
$[F_j(-)]^\perp$ and  $\left(P_n^{(j+1)}\right)^{-1/2}
[F_j(-)]^\perp$ have dimension $n-r_j^-$ while $ [F_j(+)]^\perp$
and  $\left(P_n^{(j+1)}\right)^{-1/2} [F_j(+)]^\perp$ have
dimension $n-r_j^+$. In conclusion, the subspaces $L(-)$ and
$L(+)$ defined in (\ref{important subspace -}) and (\ref{important
subspace +}), respectively, have dimensions larger than $n-r^-$
and $n-r^+$, respectively  with $r^-=\sum_{j=0}^{l-1} r_j^{-}$,
$r^+=\sum_{j=0}^{l-1} r_j^{+}$. Since the dimension of such
subspaces is large enough, we deduce that $V\cap L(-)$ and $V \cap
L(+)$ are non trivial (they have dimension at least equal to $1$),
with $V$ being any subspace reported in (\ref{max-min}) and
(\ref{min-max}). Therefore
\begin{eqnarray*} 
\lambda_k & = & \max_{\textup{dim}(V)=k} \min_{v\in V, v\neq 0} \frac{v^* A_n v}{v^* P_n v} \\
                 & \le &  \max_{\textup{dim}(V)=k} \min_{v\in V\cap L(+), v\neq 0} \frac{v^* A_n v}{v^* P_n v} \\
 & = &         \max_{\textup{dim}(V)=k} \min_{v\in V\cap L(+), v\neq 0} \prod_{j=0}^{l-1}
\frac{v^* P_n ^{(j)}v}{v^* P_n^{(j+1)} v} \\
 & \le &  \prod_{j=0}^{l-1} \beta_j =\beta,
\end{eqnarray*}
\begin{eqnarray*} 
\lambda_k & = & \min_{\textup{dim}(V)=n+1-k} \max_{v\in V, v\neq 0} \frac{v^* A_n v}{v^* P_n v}\\
& \ge &  \min_{\textup{dim}(V)=n+1-k} \max_{v\in V\cap L(-), v\neq 0} \frac{v^* A_n v}{v^* P_n v} \\
& = & \min_{\textup{dim}(V)=n+1-k} \max_{v\in V\cap L(-), v\neq 0} \prod_{j=0}^{l-1}
\frac{v^* P_n ^{(j)}v}{v^* P_n^{(j+1)} v} \\
& \ge & \prod_{j=0}^{l-1} \alpha_j =\alpha,
\end{eqnarray*}
and the proof is concluded.
\ \hfill $\bullet$

\section{The spectrum of $\{\tau_n^{-1}(f)T_n(f)\}_n$}\label{sec:main}

The main theoretical result  concerning the ill-conditioned Toeplitz sequences
and the   proposed $\tau$ preconditioners is stated below.
\begin{theorem}\label{th:main}

Let $f$ be   the generating function  of $T_n(f)$ having a single
zero at zero of order $\theta \in \mathbb{R}^+$  and let
$\tau_n(f)$ be the related $\tau$ matrix as defined in
(\ref{prec_def}). The following facts hold:
\begin{enumerate}
\item  if $\theta \in [0,2]$, then there exist constants $c, C>0$
independent of the dimension $n$, so that  $c\leq
\lambda_{i}(\tau^{-1}_n(f)T_n(f))\leq C$ for every $i,n$, i.e.,
the sequences $\{\tau_n(f)\}_n$ and $\{T_n(f)\}_n$ are spectrally
equivalent.
 \item if $\theta \in (2,\infty)$ then there exist a constant $c>0$
 and a positive number $m$ such that  $c\leq \lambda_{i}(\tau^{-1}_n(f)T_n(f))$
 for every $i,n$. Moreover, at most $m$ eigenvalues
 of this preconditioned matrix can grow to infinity.
 Hence, the essential spectral equivalence between  $\{\tau_n(f)\}_n$ and $\{T_n(f)\}_n$ holds.
 \end{enumerate}
     \end{theorem}
     \begin{proof}
     First, we recall that when $\theta=0$,  the generating
     function is strictly positive and so the spectrum  $\sigma(\tau^{-1}_n(f)T_n(f))$
     is bounded from below and above by constants $c,C>0$ independent of the dimension $n$,
     since both matrices are bounded from below and above by constants far away from zero and infinity.
     The same holds true also when $\theta=2$ since  $f$ is equivalent to $g_1(t)=2-2\cos(t)$
     in the sense that there exist  $k_1,k_2>0$ for which
     \bdm k_1 g(t)\leq f(t)\leq k_2g(t)\qquad \forall t.\edm
 Then, it is known from \cite{F.Ben:95} that for the natural $\tau$ preconditioner, $\tau^{nat}(f)$, the following inequalities
 \bdm c_1<\sigma([\tau^{nat}_n(g_1)]^{-1}T_n(f))<c_2\qquad c_1,c_2>0 \edm holds true. So  \bdm \hat{c}_1<\sigma(\tau_n^{-1}(f)T_n(f))<\hat{c}_2 \edm
 since
 \bdm \frac{x^T T_n(f)x}{x^T\tau^{nat}_n(g_1)x}= \frac{x^T T_n(f)x}{x^T\tau_n(f)x}\frac{x^T\tau_n(f)x}{x^T\tau^{nat}_n(g_1)x} \edm
and the second term on the right part is bounded far away from zero and infinity, owing to the equivalence of $g_1$ and $f$.

In the case where $\theta=4$, $f\sim g_2$ with
$g_2(t)=(2-2\cos(t))^2$. Following  again  the  above   analysis
and knowing from  \cite{F.Ben:95}  that the preconditioned  matrix
$[\tau_n^{nat}(g_2)]^{-1}T_n(f)$ has at most 2 eigenvalues growing
to infinity, we conclude that $\tau^{-1}_n(f)T_n(f)$ will also
have at most 2 eigenvalues growing to infinity as $n\rightarrow
\infty$. For the convenience of the reader we decouple the
complete proof  into the following  three  parts:
     \begin{description}
     \item [a)] the maximum eigenvalue of   $\tau^{-1}_n(f)T_n(f)$ is bounded, when $\theta \in[0,2]$;
     \item [b)] at most a constant number of eigenvalues of $\tau^{-1}_n(f)T_n(f)$ can tend to infinity, when $\theta \in (2,\infty)$;
     \item [c)] the minimum eigenvalue of  $\tau^{-1}_n(f)T_n(f)$ is bounded from below by a constant independent of $n$, when $\theta$ is a real positive number.
     \end{description}

{\bf Proof of step a)\,\,} We consider the symmetric  positive semidefinite matrix-valued function \bdm F(t)=\left(\begin{array}{cc}  1 & |t|\\ |t|& t^2 \end{array}\right)\edm
Then, the generated block Toeplitz matrix
\bdm B_{2n}(F(t))=\left(\begin{array}{cc}T_n(1) & T_n(|t|)\\ T_n(|t|)& T_n(t^2) \end{array}\right) \edm
is  positive semidefinite and  so is its Schur complement
\bdm
 S=T_n(t^2)-T_n(|t|)T_n(|t|)\geq 0\Leftrightarrow T_n(t^2)\geq T_n(|t|)T_n(|t|)\edm
 where the symbol $"\geq " $ stands for the partial ordering in the space of Hermitian matrices
 (i.e. $A\ge B$ if and only if $A$ and $B$ are both Hermitian and $A-B$ is positive semidefinite).
 Pre and post multiplying the above inequality with  the positive definite $\tau$ matrix $\tau_n(|t|^{-1})$, by the inertia law, we get
 \bdm \tau_n(|t|^{-1})T_n(t^2)\tau_n(|t|^{-1})\geq \tau_n(|t|^{-1})T_n(|t|)T_n(|t|)\tau_n(|t|^{-1}).\edm
 The matrix in the left hand side of the inequality above is similar to the preconditioned matrix $\tau_n(t^{-2})T_n(t^2)$.
 This matrix has bounded spectrum, since it corresponds to the case of $\theta=2$. Thus, taking the spectral radii in both sides, we deduce that
 \beax
 C&\geq& \rho(\tau_n(|t|^{-1})T_n(t^2)\tau_n(|t|^{-1}))\geq \tau_n(|t|^{-1})T_n(|t|)T_n(|t|)\tau_n(|t|^{-1})\\
 &=&\| \tau_n(|t|^{-1})T_n(|t|)\|_2^2\geq \rho(\tau_n(|t|^{-1})T_n(|t|))^2.
 \eeax
 Thus the maximum eigenvalue of $\tau_n(|t|^{-1})T_n(|t|)$ is bounded from above by the constant $\sqrt{C}$.

Even though this is a special case and  the considered procedure
furnishes the upper bound for the concrete case of $\theta=1$,
the idea can be easily generalized to cover any $\theta\in (0,2)$.

Let us assume that
$\rho(\tau_n(|t|^{-\theta_1})T_n(|t|^{\theta_1}))\leq C_1$ and
$\rho(\tau_n(|t|^{-\theta_2})T_n(|t|^{\theta_2}))\leq C_2$  for
some $\theta_1, \theta_2 \in [0,2]$. Let also $\hat{\theta}$ be
the arithmetic mean of $\theta_1, \theta_2$, i.e., $
\hat{\theta}=\frac{\theta_1+ \theta_2}{2}$. Then, \bdm
F(t):=\left(\begin{array}{cc}  |t|^{\theta_1} &
|t|^{\hat{\theta}}\\ |t|^{\hat{\theta}} & |t|^{\theta_2}
\end{array}\right)\geq 0 \Rightarrow
B_{2n}(F(t)):=\left(\begin{array}{cc}T_n(|t|^{\theta_1}) &
T_n(|t|^{\hat{\theta}})\\  T_n(|t|^{\hat{\theta}})&
T_n(|t|^{\theta_2})\end{array}\right)\geq 0.\edm Hence, the Schur
complement of the above block Toeplitz matrix should be positive
semidefinite, a fact that is translated into the relation \bdm
T_n(|t|^{\theta_2})\geq
T_n(|t|^{\hat{\theta}})T^{-1}_n(|t|^{\theta_1})
T_n(|t|^{\hat{\theta}}). \edm Consequently, we  pre and post
multiply both sides by the positive definite matrix
$\tau_n(|t|^{-\frac{\theta_2}{2}})$ and we use the Rayleigh
quotients to get \bea \label{rayleigh_ineq} C_2&\geq&
\frac{y^T\tau_n(|t|^{-\frac{\theta_2}{2}}) T_n(|t|^{\theta_2})
\tau_n(|t|^{-\frac{\theta_2}{2}})y}{y^Ty}\\ &\geq& \frac{y^T
\tau_n(|t|^{-\frac{\theta_2}{2}})
T_n(|t|^{\hat{\theta}})T^{-1}_n(|t|^{\theta_1})
T_n(|t|^{\hat{\theta}})\tau_n(|t|^{-\frac{\theta_2}{2}})y
}{y^Ty}.\nonumber \eea We multiply and divide the last term in the
inequality above by the quantity $z^Tz$, where
$z=\tau_n(|t|^{-\frac{\theta_1}{2}})T_n(|t|^{\hat{\theta}})\tau_n(|t|^{-\frac{\theta_2}{2}})y$.
Then, this term can be written as \bdm
\frac{z^T\tau_n(|t|^{\frac{\theta_1}{2}})T^{-1}_n(|t|^{\theta_1})\tau_n(|t|^{\frac{\theta_1}{2}})z}{z^Tz}\cdot
\frac{z^Tz}{y^Ty}. \edm For the first Rayleigh quotient we have
\bdm
\frac{z^T\tau_n(|t|^{\frac{\theta_1}{2}})T^{-1}_n(|t|^{\theta_1})\tau_n(|t|^{\frac{\theta_1}{2}})z}{z^Tz}
\geq\frac{1}{\rho(\tau_n(|t|^{-\theta_1})T_n(|t|^{\theta_1}))}\geq\frac{1}{C_1}.
\edm We substitute it into inequalities (\ref{rayleigh_ineq}) and
we infer \bea \nonumber
C_2C_1&\geq& \frac{z^Tz}{y^Ty}\\
            & = & \frac{y^T \tau_n(|t|^{-\frac{\theta_2}{2}})
            T_n(|t|^{\hat{\theta}}) \tau_n(|t|^{-\frac{\theta_1}{2}}) \tau_n(|t|^{-\frac{\theta_1}{2}})
             T_n(|t|^{\hat{\theta}}) \tau_n(|t|^{-\frac{\theta_2}{2}})y }{y^Ty}.
\eea This inequality holds also true if we take as $y$ the
eigenvector $x$  corresponding to the spectral radius
$\rho(\tau_n(|t|^{-\frac{\theta_2}{2}})    T_n(|t|^{\hat{\theta}})
\tau_n(|t|^{-\frac{\theta_1}{2}})
\tau_n(|t|^{-\frac{\theta_1}{2}})    T_n(|t|^{\hat{\theta}})
\tau_n(|t|^{-\frac{\theta_2}{2}}))$. Thus \bea
\label{rayleigh_ineq2} C_2C_1&\geq&
\rho(\tau_n(|t|^{-\frac{\theta_2}{2}})    T_n(|t|^{\hat{\theta}})
\tau_n(|t|^{-\frac{\theta_1}{2}})
\tau_n(|t|^{-\frac{\theta_1}{2}})    T_n(|t|^{\hat{\theta}})
\tau_n(|t|^{-\frac{\theta_2}{2}}))
\nonumber\\&=&\|\tau_n(|t|^{-\frac{\theta_1}{2}})
T_n(|t|^{\hat{\theta}})
\tau_n(|t|^{-\frac{\theta_2}{2}})\|_2^2\geq
\rho(\tau_n(|t|^{-\frac{\theta_1}{2}})    T_n(|t|^{\hat{\theta}})
\tau_n(|t|^{-\frac{\theta_2}{2}}))^2\\&=&
\rho(\tau_n(|t|^{-\hat{\theta}})T_n(|t|^{\hat{\theta}}))^2
\nonumber \eea and hence we have proven that the spectral radius
of the preconditioned matrix
$\tau_n(|t|^{-\hat{\theta}})T_n(|t|^{\hat{\theta}})$ has an upper
bound the constant $\sqrt{C_2C_1}$.

Starting from $\theta_1=0, \theta_2=2$, we proved the bound for
$\theta=1$. Following the very same procedure, we prove the bound
for $\theta=\frac{1}{2}$ and $\theta=\frac{3}{2}$ and so on.
Finally, we can prove the same property for every $\theta$
rational number in  $(0, 2)$ and with the important observation
that the bound does not depend on the given rational number:
indeed, when dealing with the case $\theta=1$, the bound is the
geometric mean of the bounds for $\theta=0$ and $\theta=2$ so that
it does not exceed the maximum of the two bounds; by iterating the
procedure the same observation is still true. Furthermore, since
the set of rational numbers is dense in the set of real numbers,
the same property is proven for all $\theta \in (0, 2)$, because
of the continuity of the matrices $\tau_n(|t|^{-\theta})$ and
$T_n(|t|^{{\theta}})$ with respect to the parameter $\theta$ and
because of the continuity of the spectrum with respect to the
matrix coefficients.
\ \\
{\bf Proof of step b)\,\,} We use Theorem \ref{th:gen} with $l=4$.
More precisely, taking into account that $\theta>2$, we write
$\theta=2k+r$, $k\ge 1$ integer, $r\in [0,2)$, and we  define the
following $l$ step preconditioning:
\begin{eqnarray*}
A_n & = & P_n^{(0)}=T_n(|t|^\theta), \\
P_n^{(1)} & = &   T_n((2-2\cos(t))^k|t|^r), \\
P_n^{(2)} & = &   \tau_n((2-2\cos(t))^k) T_n(|t|^r), \\
P_n^{(3)} & = &   \tau_n((2-2\cos(t))^k |t|^r), \\
P_n^{(4)} & = &   \tau_n((|t|^\theta)=P_n.
\end{eqnarray*}

Now $\left\{P_n^{(1)}\right\}$ and $\left\{A_n=P_n^{(0)}\right\}$ are spectrally equivalent and the eigenvalues of
$\left\{\left(P_n^{(1)}\right)^{-1} P_n^{(0)}\right\}$ belong to the interval $(r,R)$, with
\begin{equation}\label{bounds-toeplitz}
r=\min_{t\in [0,2\pi]} \frac{|t|^{2k}}{(2-2\cos(t))^k},
\ \ \
R=\max_{t\in [0,2\pi]} \frac{|t|^{2k}}{(2-2\cos(t))^k}.
\end{equation}
$\left\{P_n^{(2)}\right\}$ and $\left\{P_n^{(1)}\right\}$ are essentially spectrally equivalent and indeed their difference
has rank bounded by a quantity proportional to $k$, while the analysis of $\left\{P_n^{(3)}\right\}$ and $\left\{P_n^{(2)}\right\}$ reduces to the one performed in {\bf step b)}. Finally
$\left\{P_n=P_n^{(4)}\right\}$ and $\left\{P_n^{(3)}\right\}$ are spectrally equivalent and the eigenvalues of
$\left\{\left(P_n^{(4)}\right)^{-1} P_n^{(3)}\right\}$ belong to the interval $[1/R,1/r]$, with $r,R$ defined in
(\ref{bounds-toeplitz}).

The use of Theorem \ref{th:gen} leads to the desired conclusion.
\ \\
{\bf Proof of step c)\,\,} We will prove that
$\lambda_{\min}(\tau^{-1}(|t|^{\theta})T(|t|^{\theta}))>m$  with
constant $m$ independent of $n$, by proving that for every
normalized vector $z\in \mathbf{R}^n$, the corresponding Rayleigh
quotient  $\frac{z^TT_n(|t|^{\theta})z}{z^T \tau(t^{\theta})z } $
is  bounded from below by $m$. Using (\ref{prec_def}) and
(\ref{sine_matrix}) and making some simple manipulations,  we
obtain that the denominator $D$ of the above ratio can be written
as \bdm
D=z^T\tau_n(|t|^{\theta})z=(Sz)^TD(Sz)=\frac{2}{n+1}\sum_{k=1}^n\left(\frac{k\pi}{n+1}\right)^{\theta}
\left(\sum_{j=1}^n\sin{(\frac{jk\pi}{n+1})}z_j\right)^2\edm while
the numerator $N$ can be expanded as \bdm
z^TT_n(|t|^{\theta})z=\sum_{k=1}^nz_k\sum_{j=1}^n
t_{k-j}z_j=\frac{1}{2\pi}\sum_{k=1}^nz_k\sum_{j=1}^n\int_{-\pi}^{\pi}|t|^{\theta}\cos{(k-j)}t\,
dt z_j. \edm
Using the  trigonometric identity
$\cos{(a-b)}=\cos{a}\cos{b}+\sin{a}\sin{b},$  we split the above
expression in two positive terms, $C$ and $S$, where: \bdm
C=\frac{1}{2\pi}\sum_{k=1}^nz_k\sum_{j=1}^n\left(\int_{-\pi}^{\pi}|t|^{\theta}\cos{(kt)}\cos{(jt)}\,
dt
\right)z_j=\frac{1}{\pi}\int_0^{\pi}t^{\theta}(\sum_{j=1}^n\cos{(jt)}z_j)^2\,
dt \edm and the \bdm
S=\frac{1}{2\pi}\sum_{k=1}^nz_k\sum_{j=1}^n\left(\int_{-\pi}^{\pi}|t|^{\theta}\sin{(kt)}\sin{(jt)}\,
dt
\right)z_j=\frac{1}{\pi}\int_0^{\pi}t^{\theta}(\sum_{j=1}^n\sin{(jt)}z_j)^2
\,dt\edm Using the trapezoidal rule we can see that the term $S$
is strongly related to the denominator since \bdm
\frac{1}{\pi}\sum_{k=0}^n\int_{k\frac{\pi}{n+1}}^{(k+1)\frac{\pi}{n+1}}
t^{\theta}\left(\sum_{j=1}^n\sin{(jt)}z_j\right)^2\, dt\approx
\frac{1}{n+1}\sum_{k=1}^n\left(\frac{k\pi}{n+1}\right)\left(\sum_{j=1}^n\sin{(\frac{jk\pi}{n+1})}z_j\right)^2.
\edm
Following an asymptotic analysis analogous to that of Lemma 3.4 in \cite{D.Nou/P.Vas:08}, we can bound the minimum eigenvalue by a universal positive constant.
\end{proof}
\ \\

Theorem \ref{th:main} deserves a few remarks. The first observation concerns  {\bf step b)},
where the procedure for giving an upper bound to the number of the outliers is indeed an effective algorithm
that could be numerically tested. The second remark concerns the non-optimal bound that {\bf step b)} induces:
in fact, as stressed by the numerical experiments, only two 
outliers show up when $\theta\in (2,4)$. For filling the gap, we
could employ the fine technique in {\bf step a)}: however our
initial attempts allowed to give a bound on the number of outlying
singular values and this does not lead to the desired result. A
possible way for overcoming this difficulty could be the use of
the Majorization Theory, concerning the moduli of the eigenvalues
and the singular values (see \cite{R.Bha} for an elegant and rich
treatment of this theory).

\section{Numerical Experiments}\label{sec:numexp}

  In this section we report  numerical examples that were conducted in order to point out
 the efficiency of the proposed preconditioners and to confirm
 the validity of the presented theory.    The experiments were carried out using  Matlab and in
 the examples where a linear system is involved the  righthand side vector is chosen as  $(1 ~1~
\cdots~1)^T$. Although we have run also our examples with the righthand side being random
vectors (and the results are essentially of the same type), we adopt the previous choice
in order to present a fair comparison with the  methods and numerical tests given in the relevant literature.
In all cases, the zero vector was chosen  as initial guess for the PCG method and the stopping criterion was the inequality
$\frac{\|r^{(j)}\|_2}{\|r^{(0)}\|_2} \leq 10^{-7}$, where
$r^{(j)}$ is the residual vector in the $j$th iteration.

In Figure \ref{fig:1} we give a snapshot of  the asymptotical
behavior of  the eigenvalues of $\tau_n^{-1}(f)T_n(f)$  where
$f(t)=|t|^3$ and the matrix $\tau_n(f)$ is constructed as in
(\ref{prec_def}). It is clear, and as the theory predict, that
from below the minimum eigenvalue of the sequence
$\{\tau_n^{-1}(f)T_n(f)\}_n$ is bounded by a constant, the main
mass of them is clustered around one while at most two of them
seem to tend to infinity.
\begin{center}
\begin{figure}[!h]
\begin{center}
\includegraphics[width=4.1in, height=6cm]{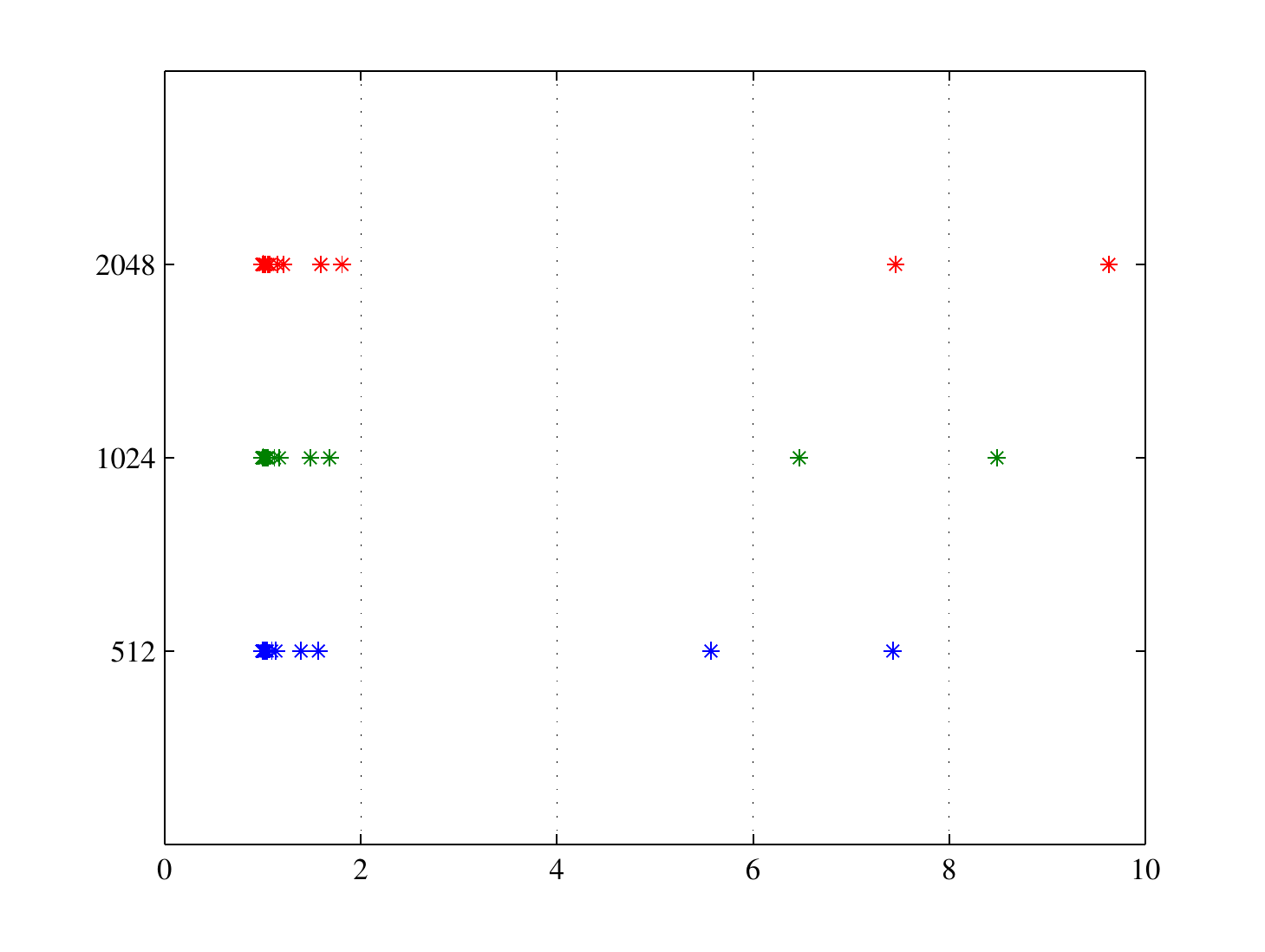} \caption{ \small{ Spectrum of
$\tau_n(f)^{-1}T_n(f)$, where $f(t)=|t|^3$} }
\label{fig:1}
\end{center}
\end{figure}
\end{center}

In the next tables,  we display the performance of our proposed
preconditioner   applied to various ill-conditioned Toeplitz
systems. In all  cases, the coefficient matrix is generated by a
function with a unique zero at zero of  non-even order $\theta$.
As we have mentioned in the introduction, for these cases there is
no suitable optimal PCG method. A non-optimal proposal  is
presented in \cite{S.Ser:95:Cal} where the preconditioner is the
band Toeplitz matrix generated by the trigonometric polynomial
$(2-2\cos{(t)})^{2k}$ where the number $k$ is such that the
distance $|2k-\theta|$ is minimum. Following the convergence
analysis of the PCG method (see e.g. \cite{O.Axe/G.Lin:86}) and
the spectral behavior of the aforementioned  preconditioner
analyzed extensively in \cite{F.Ben/S.Ser:99}, we can easily
conclude that in our case it is better  to overestimate $\theta$
rather than to underestimate it. The reason is that in the latter
case $\it{O}(n)$ eigenvalues of the preconditioned matrix  will
tend to infinity, while in the first case  $\it{O}(n)$ eigenvalues
will tend to zero. We denote the  preconditioner proposed in
\cite{S.Ser:95:Cal} as $S$ while our preconditioner is shortly
indicated with the symbol $\tau$. For our experiments we have
chosen the following generating functions: \beax f_1(t)=|t|, \quad
f_2(t)=|t|^{\frac{7}{2}}, \quad f_3(t)=|t|^3, \quad
f_4(t)=|t|^{\frac{9}{2}}. \eeax The corresponding iterations are
reported on Tables \ref{tab:1}, \ref{tab:2}, \ref{tab:3},
\ref{tab:4}. For all the examples,  we remark that  the
unpreconditioned CG method requires a  number of iterations
exceeding 1000, even for moderate matrix-sizes like  $n=512$.

\begin{table}[h!]
\begin{center}
\caption{Number of iterations for $f(t)=|t|$, the extreme eigenvalues of $P_n^{-1}(f)T_n(f)$ and the number of unbounded  eigenvalues. } \label{tab:1}
\begin{tabular}{||c|c|c||c|c|c||}\hline\hline
  n &  S & $\tau$  &$\lambda_{\min}$ &$\lambda_{\max}$& $\sharp\{\lambda_i(P)\}>2 $ \\\hline
  256 & 33 & 6 &  0.61 & 1.04 & 0 \\
  512 & 44  & 6 & 0.60 &1.04 & 0  \\
  1024 & 63 & 6 & 0.59 & 1.04 & 0  \\
  2048 & 89 & 6 & 0.59 & 1.04 & 0 \\
  4096 &124  & 7 & 0.58 & 1.04 & 0 \\\hline\hline
\end{tabular}
\end{center}
\end{table}

\begin{table}[h!]

\begin{center}
\caption{Number of iterations for $f(t)=|t|^{3}$, the extreme eigenvalues of $P_n^{-1}(f)T_n(f)$ and the number of unbounded  eigenvalues. } \label{tab:2}
\begin{tabular}{||c|c|c||c|c|c||}\hline\hline
  n &  S & $\tau$  &$\lambda_{\min}$ &$\lambda_{\max}$& $\sharp\{\lambda_i(P)\}>2 $ \\\hline
  256 & 9& 34 & 1  & 6.4 &  2\\
  512 & 9  & 51 &1  & 7.4&  2 \\
  1024 & 9 & 78 & 1 & 8.5 & 2  \\
  2048 & 10 & 118 & 1 & 9.8  & 2  \\
  4096 &10  & 179 &1  & 11.2 & 2 \\\hline\hline
\end{tabular}
\end{center}
\end{table}

\begin{table}[h!]

\begin{center}
\caption{Number of iterations for $f(t)=|t|^{\frac{7}{2}}$, the extreme eigenvalues of $P_n^{-1}(f)T_n(f)$ and the number of unbounded  eigenvalues. } \label{tab:3}
\begin{tabular}{||c|c|c||c|c|c||}\hline\hline
  n &  S & $\tau$  &$\lambda_{\min}$ &$\lambda_{\max}$& $\sharp\{\lambda_i(P)\}>2 $ \\\hline
  256 & 20 & 9 &  1 & 32.2 & 2 \\
  512 & 24  & 10 & 1 &46.5 & 2  \\
  1024 & 31 & 10 & 1 & 66.9 & 2  \\
  2048 & 40 & 11 & 1 & 96.3 & 2  \\
  4096 & 52 & 11 & 1 & 137.8 & 2 \\\hline\hline
\end{tabular}
\end{center}
\end{table}

\begin{table}[h!]

\begin{center}
\caption{Number of iterations for $f(t)=|t|^{9/2}$, the extreme eigenvalues of $P_n^{-1}(f)T_n(f)$ and the number of unbounded  eigenvalues. } \label{tab:4}
\begin{tabular}{||c|c|c||c|c|c||}\hline\hline
  n &  S & $\tau$  &$\lambda_{\min}$ &$\lambda_{\max}$& $\sharp\{\lambda_i(P)\}>2 $ \\\hline
  256 &  45 & 10 & 0.77  &$1.1\times 10^3$  & 2 \\
  512 &  62 & 11 & 0.74 &$3.0\times 10^3$  &  2 \\
  1024 & 86 & 13 & 0.72 &$8.5\times 10^3$  &  2 \\
  2048 & 119 &14  & 0.70 &$2.4\times 10^4$  & 2  \\
  4096 & 165 & 14 & 0.69 & $6.8\times 10^4$ &  \\\hline\hline
\end{tabular}
\end{center}
\end{table}

An important point is that in all considered cases,  we only
observed 2 outliers, showing that there is room for theoretical
improvement in Theorem \ref{th:main}.

\section{Concluding remarks}\label{sec:final}

In previous works, the spectral equivalence of the matrix
sequences  $\{{\tau}_n(f)\}_n $ and $\{T_n(f) \}_n$ was proven
under the assumption that the symbol $f$ has a zero of order $2$
at zero: furthermore, if the order $\theta$ is an even number
larger than $2$, the essential spectral equivalence was proved.
Here we have expanded the previous result to any positive order
$\theta$, by showing that the spectral equivalence holds for
$\theta\le 2$ and the essential spectral equivalence can be proven
for every $\theta>2$.

A possible line of further research could concern   extending the
validity of the proposed idea also to other trigonometric matrix
algebras, (e.g., the circulant algebra) and the  multi-level case.
Obviously,  more difficulties are expected on this directions due
to the  facts that the $\tau$ algebra is closer in a rank sense to
the Toeplitz structure, when the generating function of the latter
is a even trigonometric polynomial, and, due to the negative
results that hold in the multidimensional case (see
\cite{S.Ser/E.Tyr:99}, \cite{D.Nou/S.Ser/P.Vas:04:TCS}).
Furthermore, concerning Theorem \ref{th:main}, the proof technique
used in step {\bf a)}  is rather precise for $\theta\in [0,2]$,
but it did not work for larger values of $\theta$: a further
investigation in this direction would be useful in order to prove
a precise bound on the number of outliers, since Theorem
\ref{th:gen} used in step {\bf b)} provides a non-optimal
estimate, as suggested  by the numerical tests.

\bibliographystyle{siam}
\bibliography{totbib_181009}

\end{document}